\documentclass[12pt]{amsart}

\usepackage{amsmath}
\usepackage{amssymb}  % defines \nmid, for example
\usepackage{latexsym} % for \Box
\usepackage{comment}
\usepackage{url}
\usepackage{fullpage,url,amssymb,amsmath,amsthm,amsfonts,mathrsfs}
\usepackage[usenames,dvipsnames]{color}
\usepackage[pagebackref = true, colorlinks = true, linkcolor = blue, citecolor = Green]{hyperref}
\usepackage[alphabetic,lite]{amsrefs}
\usepackage{enumitem}
\usepackage{amscd}   % for commutative diagrams
\usepackage[all, cmtip]{xy} % for complicated commutative diagrams
\usepackage{xfrac}
\usepackage[T1]{fontenc}

\DeclareFontEncoding{OT2}{}{} % to enable usage of cyrillic fonts
\newcommand{\textcyr}[1]{%
 {\fontencoding{OT2}\fontfamily{wncyr}\fontseries{m}\fontshape{n}\selectfont #1}}

\usepackage[all]{xy}
\usepackage{fullpage}
\newcommand{\Sha}{{\mbox{\textcyr{Sh}}}}

\usepackage{color} 
\newcommand{\defi}[1]{\textsf{#1}} % for defined terms

\def\act#1#2%
  {\mathop{}%
   \mathopen{\vphantom{#2}}^{#1}%
   \kern-3\scriptspace%
   #2}

\newcommand{\Z}{{\mathbb Z}}
\newcommand{\N}{{\mathbb N}}
\newcommand{\Q}{{\mathbb Q}}

\newcommand{\F}{{\mathbb F}}
\newcommand{\A}{{\mathbb A}}
\newcommand{\PP}{{\mathbb P}}
\newcommand{\G}{{\mathbb G}}

\newcommand{\Kbar}{{\overline{K}}}
\newcommand{\Cbar}{{\overline{C}}}
\newcommand{\kbar}{{\overline{k}}}

\newcommand{\Xbar}{{\overline{X}}}

\newcommand{\Jbar}{{\overline{J}}}

\newcommand{\al}{{\alpha}}
\newcommand{\calC}{{\mathcal C}}

\newcommand{\calO}{{\mathcal O}}
\newcommand{\OO}{{\mathcal O}}

\DeclareMathOperator{\Gal}{Gal}

\DeclareMathOperator{\Br}{Br}

\DeclareMathOperator{\Pic}{Pic}
\DeclareMathOperator{\Alb}{Alb}
\DeclareMathOperator{\Jac}{Jac}
\DeclareMathOperator{\HH}{H}

\DeclareMathOperator{\Spec}{Spec}

\DeclareMathOperator{\GL}{GL}

\DeclareMathOperator{\PS}{PS}
\DeclareMathOperator{\Char}{char}

\newcommand{\kk}{\mathbf{k}}

\newcommand{\pr}{\operatorname{pr}}

\newtheorem{Theorem}{Theorem}[section]
\newtheorem{Lemma}[Theorem]{Lemma}
\newtheorem{Proposition}[Theorem]{Proposition}
\newtheorem{Corollary}[Theorem]{Corollary}
\newtheorem{Definition}[Theorem]{Definition}
\newtheorem{Example}[Theorem]{Example}
\newtheorem{Remark}[Theorem]{Remark}
\newtheorem{Remarks}[Theorem]{Remarks}
\newtheorem{Conjecture}[Theorem]{Conjecture}

\numberwithin{equation}{section}

\begin{document}

\title{The $d$-primary Brauer-Manin obstruction for curves}

\author{Brendan Creutz}
\address{School of Mathematics and Statistics, University of Canterbury, Private Bag 4800, Christchurch 8140, New Zealand}
\email{brendan.creutz@canterbury.ac.nz}
\urladdr{http://www.math.canterbury.ac.nz/\~{}b.creutz}

\author{Bianca Viray}
\address{University of Washington, Department of Mathematics, Box 354350, Seattle, WA 98195, USA}
\email{bviray@math.washington.edu}
\urladdr{http://math.washington.edu/\~{}bviray}

\author{Jos\'e Felipe Voloch}
\address{School of Mathematics and Statistics, University of Canterbury, Private Bag 4800, Christchurch 8140, New Zealand}
\email{felipe.voloch@canterbury.ac.nz}
\urladdr{http://www.math.canterbury.ac.nz/\~{}f.voloch}

\thanks{The second author was partially supported by NSF CAREER grant DMS-1553459 and a University of Canterbury Visiting Erskine Fellowship. The third author
acknowledges support from the Simons Foundation via the Collaboration Grant \#234591.
}

\begin{abstract}
For a curve over a global field we consider for which integers $d$ the $d$-primary part of the Brauer group can obstruct the existence of rational points. We give examples showing it is possible that there is a $d$-primary obstruction for infinitely many relatively coprime $d$, and also an example where the odd part of the Brauer group does not obstruct although there is a Brauer-Manin obstruction. These examples demonstrate that a slightly stronger form of a conjecture of Poonen is false, despite being supported by the same heuristic.
\end{abstract}

\maketitle
%

%%%%%%%%%%%%%%%%%%%
\section{Introduction}%%%%%%%%
%%%%%%%%%%%%%%%%%%%
  Let $k$ be a global field, $\A_k$ its ring of adeles and let
$X$ be a smooth projective and geometrically integral variety over $k$.  
Manin observed that any adelic point  $(P_v) \in X(\A_k)$ that is approximated by a $k$-rational point must satisfy relations imposed by elements of $\Br X := \HH^2_\text{\'et}(X, \G_m)$, the Brauer group of $X$~\cite{Manin1971}. {More precisely, }{there is a pairing $X(\A_k) \times \Br X \to \Q/\Z$ such that, for each element $\alpha \in \Br X$, the set $X(\A_k)^{\alpha}$ of adelic points pairing trivially with $\alpha$ is closed and contains the $k$-rational points of $X$. } In particular, for any subset $B \subset \Br X$ the implication
    \begin{equation}\label{eq:Brset}
        X(\A_k)^{B} := \bigcap_{\alpha\in B}X(\A_k)^{\alpha} = \emptyset
        \quad \Longrightarrow \quad
        X(k) = \emptyset.
    \end{equation}
	always holds. When $X(\A_k)^{\Br} = \emptyset$ one says there is a Brauer-Manin obstruction to the existence of rational points. {Here and below $X(\A_k)^{\Br}$ is used to denote $X(\A_k)^{\Br X}$.}

It is an outstanding open question, first raised by Skorobogatov \cite[p. 133]{TorsorsAndRationalPoints} in the number field case, whether this is the only obstruction for $X = C$ a curve. Poonen~\cite{Poonen} and Stoll~\cite{Stoll07} have conjectured an affirmative answer to this question. This has been proved for curves over function fields,
under additional but quite general hypotheses \cite{PoonenVoloch}. 

Computation of the set $C(\A_k)^{\Br}$ is quite difficult, in part because the Brauer group of a curve over a global field is quite large. Modulo constant algebras there are infinitely many elements of order $n$ for every integer $n$. However, in any case where $C(\A_k)^{\Br} = \emptyset$, there is an integer $d$ such that $C(\A_k)^{\Br C[d]} = \emptyset$, where $\Br C[d]$ denotes the $d$-torsion subgroup. In fact, more is true. A simple compactness argument shows that, if $C(\A_k)^{\Br} = \emptyset$, then there exists a finite subgroup $B$ of $\Br C$ with $C(\A_k)^{B} = \emptyset$ (see, e.g.,~\cite[Lemma 4.8]{CV17}). In this paper, we investigate what one can determine about such a subgroup \textit{a priori} and, in particular, which integers $d$ allow for a \defi{$d$-primary Brauer-Manin obstruction}, i.e., there is such a subgroup $B \subset \Br C$ consisting of {elements that are annihilated by some power of $d$}.  Knowledge of such subgroups can facilitate computation Brauer-Manin obstructions in practice and may shed light on the nature of the Brauer-Manin obstruction itself.  In addition, if the conjecture is false, having such an a-priori hold on a subgroup $B$ opens the possibility of producing an explicit counterexample.

If we assume finiteness of Tate-Shafarevich groups of abelian varieties, 
then it follows from Manin's arguments in~\cite{Manin1971} 
% then it follows from the arguments in Manin's original paper~\cite{Manin1971} which introduced the Brauer-Manin obstruction 
that any Brauer-Manin obstruction on a torsor under an abelian variety can be detected from the $P$-primary subgroup of the Brauer group, where $P$ is the period of the torsor in question {(i.e. its order in the Weil-Ch\^atelet group)}.  A stronger, unconditional statement along these lines is proved in~\cite{CV17}: if there is an $d$-primary Brauer-Manin obstruction for some integer $d$, then there is a $\gcd(d,P)$-primary Brauer-Manin obstruction. Similar results are achieved in the case of Kummer varieties~\cites{CV17, SZ16}, bielliptic surfaces (assuming finiteness of certain Tate-Shafarevich groups)~\cite{CV17}, and fibrations of torsors under abelian varieties over projective space~\cite{Nakahara} as well as minimal del Pezzo surfaces (by combining~\cite{Corn07,Nakahara,SD-BrauerCubic}; see~\cite{CV17}). Note that all cases considered thus far have non-maximal Kodaira dimension.

In this paper, we turn our attention to curves of genus at least $2$, the simplest examples of varieties of maximal Kodaira dimension. Our results show that the situation here is quite different. We prove that there exist curves with a $d$-primary Brauer-Manin obstruction for infinitely many pairwise coprime integers $d$ (Theorem~\ref{thm:GivenEInfinitelyManyd}). Given the well known connection between abelian descent and the Brauer-Manin obstruction, we even expect there to be locally soluble curves of general type with a $p$-primary Brauer-Manin obstructions for every prime $p$. We show, however, that this is not the case for {all} curves {without rational points}. For any prime $p$ we construct a pointless curve of general type with no $p$-primary Brauer-Manin obstruction (Theorem~\ref{thm:Brell}). We {also} construct a pointless genus $3$ curve $C$ {with a zero-cycle of degree $1$} over $\Q$ for which there is no odd Brauer-Manin obstruction (Theorem~\ref{thm:Brodd}). This construction yields an infinite tower of non-isomorphic \'etale coverings of $X$ in which all covers are counterexamples to the Hasse principle (Remark~\ref{rem:towers}). {Assuming finiteness of Tate-Shafarevich groups of elliptic curves, such towers cannot exist above genus $1$ curves.}

These examples have an interesting interpretation in relation to Poonen's heuristic~\cite{Poonen} for rational points on curves, one of the key pieces of evidence in support of (a strong form of) the conjecture that Brauer-Manin is the only obstruction for curves. Poonen shows that if, for all places $v$ of $k$ of good reduction for the Jacobian $J$ of $C$, the set of 
rational points of $C$ over the residue field $\F_v$ viewed inside the group $J(\F_v)$ behaves like a random subset of the same cardinality and the order of $J(\F_v)$ is as smooth as a typical integer of the same size, then the conjecture holds 
with probability one. We show that the implication of Poonen's heuristic goes through if, instead of $J(\F_v)$, we take the quotient of $J(\F_v)$ by its $d$-primary subgroup for an integer $d$ (Theorem~\ref{thm:primetodHeuristic}). However, we show that the conclusion is false in examples such as those mentioned in the paragraph above. This suggests one should be careful putting too much faith in such heuristics.

Constant curves over global function fields are excluded by the hypotheses in \cite{PoonenVoloch} and are known to exhibit somewhat unusual behavior, e.g., the failure of Mordell's conjecture. It is thus unclear in this case whether one should even expect Brauer-Manin to be the only obstruction. In view of this, we have included constant curves in our investigations, producing several relevant examples and general results (e.g., Theorem~\ref{thm:constant} and Proposition~\ref{prop:PSconverse}).

\subsection*{Notation}
A \defi{nice variety} over a field $K$ is a smooth, projective and geometrically integral $K$-variety. We will use \defi{curve} to mean a nice variety of dimension $1$. For a $K$-scheme $X$ we use $\Xbar$ to denote the base change of $X$ to a separable closure $\Kbar$ of $K$ and $\kk(X)$ to denote the sheaf of total quotient rings. Note that for $X$ integral, $\kk(X)$ can be identified with the function field of $X$. We use $\HH^i(X,\bullet)$ to denote flat cohomology groups, and abbreviate $\HH^i(\Spec K,\bullet)$ to $\HH^i(K,\bullet)$. The Brauer group of a scheme $X$ is $\Br X := \HH^2(X,\G_m)$. This agrees with the \'etale cohomology group $\HH^2_\text{\'et}(X,\G_m)$ as defined above \cite[Theorem III.3.9]{Milne}. We use $\Br_1 X$ to denote the kernel of the map $\Br X \to \Br \Xbar$ induced by base change. %\bianca{Define $\Alb, \Pic, \Jac$?}

Throughout the paper $k$ denotes a global field. The set of places of $k$ is denoted $\Omega_k$. For $v \in \Omega_k$, the corresponding completion is $k_v$ and, when $v$ is finite, the residue field is $\F_v$. We use $\A_k$ to denote the ring of adeles of $k$. For $X$ a curve (resp. an abelian variety) over $k$ we use $X_{\F_v}$ to denote the special fiber of the minimal proper regular model (resp. N\'eron model) of $X\times_k k_v$. When $X_{\F_v}$ is a nice $\F_v$-variety we often abuse notation by writing $X(\F_v)$ in place of $X_{\F_v}(\F_v)$.

For an integer $d$ and an abelian group $G$ (or group scheme), let $G[d^\infty]$ and $G[d^\perp]$ denote, respectively, the $d$-primary and prime-to-$d$ torsion subgroups. Recall that a \defi{supernatural number} is a formal product $\prod p^{n_p}$ ranging over all prime numbers $p$. We extend the notation $G[d^\infty]$ and $G[d^\perp]$ to supernatural numbers $d$ in the obvious way. Namely $G[d^\infty] = \bigcup_{d_0 \in \N, d_0 \mid d} G[d_0^\infty]$ and $G[d^\perp] = \bigcap_{d_0 \in \N, d_0 \mid d}G[d_0^\perp]$. For any supernatural $d$ and any torsion abelian group $G$ we have a splitting ${G \simeq G[d^\infty] \times G[d^\perp]}$. We use $\pr_{d^\infty}$ and $\pr_{d^\perp}$ to denote the corresponding projections.

%%%%%%%%%%%%%%%%%%%%%%%%%%%%%%%%%%%%%%%%%%%%%%%%%%%%%%%%%%%%%%%%%%%%%%%%%%%%%%%%
\section*{Acknowledgements}
%%%%%%%%%%%%%%%%%%%%%%%%%%%%%%%%%%%%%%%%%%%%%%%%%%%%%%%%%%%%%%%%%%%%%%%%%%%%%%%%
    This work was heavily based on discussions that took place at University of Canterbury while the second author was a visiting Erskine Fellow and at the Rational Points 2017 workshop.  We thank the Erskine Programme for its support and Michael Stoll for his support as well as for a number of helpful comments.

%%%%%%%%%%%%%%%%%%%%%%%%%%%%%%%%%%%%%%%%%%%%%%%%%%%%%%%%%%%%%%%%%%%%%%%%%%%%%%%%
\section{Preliminaries}
%%%%%%%%%%%%%%%%%%%%%%%%%%%%%%%%%%%%%%%%%%%%%%%%%%%%%%%%%%%%%%%%%%%%%%%%%%%%%%%%

%%%%%%%%%%%%%%%%%%%%%%%%%%%%%%%%%%%%%%%%%%%%%%%%%%%%%%%%%%%%%%%%%%%%%%%%%%%%%%%%
\subsection{Relating the Brauer-Manin obstruction, finite abelian descent, and divisibility on abelian varieties}
%%%%%%%%%%%%%%%%%%%%%%%%%%%%%%%%%%%%%%%%%%%%%%%%%%%%%%%%%%%%%%%%%%%%%%%%%%%%%%%%

	Let $X$ be a nice variety over a global field $k$ with Albanese variety $A := \Alb^0_X$ and Picard variety $\hat{A} = \Pic^0_X$. For an integer $d \ge 2$, we define a \defi{$d$-covering} of $X$ to be an $X$-torsor under $\hat{A}[d]$ whose \defi{type} (cf. \cite[Section 2.3]{TorsorsAndRationalPoints}) is the inclusion $\lambda_d \colon \hat{A}[d](\kbar) \hookrightarrow \hat{A}(\kbar) = \Pic^0_X(\kbar) \to \Pic \Xbar$. We define $X(\A_k)^{d-\text{ab}}$ to be the set of adelic points on $X$ that lift to an adelic point on some $d$-covering of $X$.  {Recall that $X(k) \subset X(\A_k)^{d-\textup{ab}}.$}

	From the Hochschild-Serre spectral sequence we have a map $\Br_1X \to \HH^1(k,\Pic\Xbar)$ which induces an isomorphism $\Br_1X/\Br_0X \simeq \HH^1(k,\Pic\Xbar)$, {since $k$ a global field implies $\HH^3(k,\kbar^\times) = 0$ (cf. \cite[8.3.11(iv) and 8.3.12]{CON}).} The map $\lambda_d$ induces maps
	\begin{equation}\label{eq:lambdastar}
		\HH^1(k,\hat{A}[d]) \to \HH^1(k,\hat{A}) \to \HH^1(k,\Pic\Xbar)\,.
	\end{equation}
	We define $\Br_{1/2,d}X$ and $\Br_{1/2}X$ to be the subgroups of $\Br_1X$ whose images in $\HH^1(k,\Pic\Xbar)$ lie in the images of $\HH^1(k,\hat{A}[d])$ and $\HH^1(k,\hat{A})$ under~\eqref{eq:lambdastar}, respectively. {The notation $\Br_{1/2}X$ was first introduced, in the number field case, in \cite{Stoll07}.} To simplify the notation we write $X(\A_k)^{\Br_{1/2}}$ in place of $X(\A_k)^{\Br_{1/2}X}$ and similarly for $X(\A_k)^{\Br_{1/2,d}}$, etc.

    \begin{Proposition}\label{prop:dcovA}
        Let $A$ be an abelian variety over a global field $k$. Then 
        $A(\A_k)^{\Br_{1/2,d}}= A(\A_k)^{d-\textup{ab}}$\,.
    \end{Proposition}

    \begin{proof}
        For $k$ a number field, this is a special case of \cite[Theorem 6.1.2(a)]{TorsorsAndRationalPoints}. This can be extended to global fields of positive characteristic as follows. The proof in the number field case involves three steps, as outlined at the top of page 115 of op. cit. The first step is to show the existence of an $A$-torsor of type $\lambda_d$. In our situation this is guaranteed since multiplication by $d$ endows $A$ with the structure of a $d$-covering. The second step of the proof (beginning on page 119 of op. cit.) is to prove the containment $A(\A_k)^{\Br_{1/2,d}} \subset A(\A_k)^{d-\text{ab}}$, assuming the existence of an $A$-torsor of type $\lambda_d$. This follows from local Tate duality and the Poitou-Tate exact sequence. The proof carries through verbatim in the function field case, provided one uses the local duality statement \cite[III.6.10]{MilneADT} in place of \cite[I.2.3]{MilneADT} and the Poitou-Tate sequence \cite[Theorem 5.1]{Cesnavicius} in place of \cite[I.4.20]{MilneADT}. The third step of the proof is to show the containment  $A(\A_k)^{d-\text{ab}} \subset A(\A_k)^{\Br_{1/2,d}}$. This follows from functoriality of the Brauer pairing and the fact that for any $A$-torsor $\pi:Y \to A$ of type $\lambda_d$ the induced map $\pi^*$ annihilates $\Br_{1/2,d}A/\Br_0A$.
    \end{proof}

    \begin{Corollary}\label{cor:AVdiv}
        Let $A$ be an abelian variety over a global field $k$ and let $(P_v) \in A(\A_k)$ be an adelic point. If there exists $Q \in A(k)$ such that $((P_v)-Q) \in dA(\A_k)$, then $(P_v) \in A(\A_k)^{\Br_{1/2,d}}$. Conversely, if $s$ is a positive integer such that $d^s\Sha(k,A)[d^\infty] = 0$ and $(P_v) \in A(\A_k)^{\Br_{1/2,d^n}}$ with $n -s \ge 0$, then there exists $Q \in A(k)$ such that $((P_v) -Q) \in d^{n-s}A(\A_k)$.
    \end{Corollary}
    \begin{Remark}\label{rem:Akbar}
	Let $A(\A_k)_{\bullet} = \prod \pi_0(A(k_v))$, where $\pi_0$ denotes the group of connected components. Note that $\pi_0(A(k_v)) = A(k_v)$ if $v$ is not archimedean. It is known that the closure of the image of $A(k)$ in $A(\A_k)_\bullet$ is its profinite completion. From this and the previous lemma it follows that $\overline{A(k)} = A(\A_k)_\bullet^{\Br}$ when $\Sha(k,A)$ is finite. See \cite[Remark 4.5]{PoonenVoloch}.
    \end{Remark}
    \begin{proof}
        If $(P_v) - Q \in dA(\A_k)$, then $(P_v)$ lifts to the $d$-covering $A \ni x \mapsto dx + Q \in A$, so $(P_v) \in A(\A_k)^{d-\textup{ab}}$, which equals $A(\A_k)^{\Br_{1/2,d}}$ by Proposition~\ref{prop:dcovA}. Now suppose $(P_v)$ is as in the second statement of the lemma. By Proposition~\ref{prop:dcovA}, $(P_v)$ survives $d^n$-descent, so there is a $d^n$-covering $\pi\colon Y \to A$ such that $(P_v) = \pi(Q_v)$ for some $(Q_v) \in Y(\A_k)$. The map $\pi$ must factor as $Y \to Y' \to A$ where $Y' \to A$ is a $d^{n-s}$-covering. Moreover, the assumption that $d^s\Sha(k,A) = 0$ implies that the class of $Y'$ in $\Sha(k,A)$ is trivial. Hence there is an isomorphism $Y' \simeq A$ identifying the map $Y' \to A$ with the map $A \to A$ given by $x \mapsto d^{n-s}x + Q$ for some $Q \in A(k)$. In particular, $((P_v) - Q) \in d^{n-s}A(\A_k)$. 
    \end{proof}

    %%%%%%%%%%%%%%%%%%%%%%%%%%%%%%%%%%%%%%%%%%%%%%%%%%%%%%%%%%%%%%%%%%%%%%%%%%%%
    \subsection{Relating the Brauer-Manin obstruction, finite abelian descent, and divisibility on curves}\label{subsec:curves}
    %%%%%%%%%%%%%%%%%%%%%%%%%%%%%%%%%%%%%%%%%%%%%%%%%%%%%%%%%%%%%%%%%%%%%%%%%%%%

 {This section is highly influenced by \cite{Stoll07}. In the number field case the results of this section can be found in or easily deduced from \cite[Sections 6 and 7]{Stoll07}.}
        Suppose $C/k$ is a nice curve of genus at least $1$ over a global field $k$. Let $J := \Jac(C) = \Alb^0_C = \Pic_C^0$ be its Jacobian. Suppose further that there exists a zero cycle $z$ of degree $1$ on $C$ and that this is used to define an embedding $\iota:C \to J$ by the rule $x \mapsto [x-z]$.

        \begin{Lemma}\label{lem:iota}
            The map $\iota$ induces surjective morphisms
            \begin{enumerate}
                \item $\Br J \to \Br C$,
                \item $\Br_{1/2}J \to \Br C$, and 
                \item $\Br_{1/2,d}J \to \Br C[d]$.
            \end{enumerate}
        \end{Lemma}

        \begin{proof}
            The existence of a zero-cycle of degree $1$ implies  that the exact sequence of Galois modules $0 \to \Pic^0\Cbar \to \Pic\Cbar \to \Z \to 0$ is split and so induces an isomorphism $\HH^1(k,\Pic^0\Cbar) \simeq \HH^1(k,\Pic\Cbar)$. Then we have a commutative diagram {(with non-exact rows)},
            \[
                \xymatrix{
                    \HH^1(k,\Pic^0\Jbar)\ar[r]\ar@{=}[d]^{\iota^*} & \HH^1(k,\Pic\Jbar) \ar[r]\ar[d]^{\iota^*} & \Br_1J/\Br_0 J \ar[d]^{\iota^*}\\
                    \HH^1(k,\Pic^0\Cbar) \ar[r] &  \HH^1(k,\Pic\Cbar) \ar[r]& \Br_1 C/\Br_0C 
                }
            \]

            The maps in the bottom row are all isomorphisms, so the vertical maps must all be surjective. By definition $\Br_{1/2}J \subset \Br J$ is the subgroup whose image lies in the image of $\HH^1(k,\Pic^0\Jbar)$. Furthermore, $\Br_1 C = \Br C$ by Tsen's theorem {and the existence of a zero-cycle of degree $1$ yields an isomorphism $\Br_0 J \to \Br_0 C$}, so we have the surjectivity of (2). This shows that (1) is also surjective. Statement (3) follows from the surjectivity of the map $\HH^1(k,J[d]) \to \HH^1(k,J)[d] \simeq (\Br_1C/\Br_0C)[d]$.
        \end{proof}

        Corollary~\ref{cor:AVdiv}, Lemma~\ref{lem:iota} and functoriality of the Brauer pairing yield the following.

        \begin{Corollary}\label{cor:curved-ab}
            For any integer $d$, $C(\A_k)^{d-\text{ab}} = C(\A_k)^{\Br C[d]}$.
        \end{Corollary}		

        \begin{Corollary}\label{Cor:BMCurvesBMJac}
            Let $(P_v) \in C(\A_k)$. Then $(P_v) \in C(\A_k)^{\Br}$ if and only if $\iota(P_v) \in J(\A_k)^{\Br}$.
        \end{Corollary}

        \begin{Corollary}
            Let $(P_v) \in C(\A_k)$ be an adelic point. If there exists $Q \in J(k)$ such that $(\iota(P_v)-Q) \in dJ(\A_k)$, then $(P_v) \in C(\A_k)^{\Br C[d]}$. Conversely, if $s$ is a positive integer such that $d^s\Sha(k,J)[d^\infty] = 0$ and $(P_v) \in C(\A_k)^{\Br C[d]}$ with $n -s \ge 0$, then there exists $Q \in J(k)$ such that $(\iota(P_v) -Q) \in d^{n-s}J(\A_k)$.
        \end{Corollary}

        \begin{Corollary}\label{cor:d-div}
            If for each $n \ge 1$ there exists $Q_n \in J(k)$ such that $(\iota(P_v) - Q_n) \in d^nJ(\A_k)$, then $(P_v) \in C(\A_k)^{\Br C[d^\infty]}$. If $\Sha(k,J)[d^\infty]$ is finite, then the converse holds.
        \end{Corollary}

    %%%%%%%%%%%%%%%%%%%%%%%%%%%%%%%%%%%%%%%%%%%%%%%%%%%%%%%%%%%%%%%%%%%%%%%%%%%%
    \subsection{Relating the $d$-primary Brauer-Manin and  Poonen-Scharaschkin sets on curves}\label{subsec:PS}
    %%%%%%%%%%%%%%%%%%%%%%%%%%%%%%%%%%%%%%%%%%%%%%%%%%%%%%%%%%%%%%%%%%%%%%%%%%%%

        Let $C$ be a curve over a global field $k$ with $z$ a zero-cycle of degree $1$ on $C$.  Let $\iota\colon C \hookrightarrow J$ be the embedding, $x\mapsto [x - z]$. For each prime $v$ of $k$ that is of good reduction for $C$ (and hence $J$) use $r_v \colon J(k_v) \to J(\F_v)$ to denote the reduction map and let $r \colon J(\A_k) \to \prod_{v \text{ good}}J(\F_v)$ be the product of the $r_v$ over all primes of good reduction.
        
        For a supernatural number $d$, consider the following commutative diagram,
        \begin{equation}\label{eq:PS}
        \displaystyle
            \xymatrix{
            C(k) \ar[r]^\iota \ar[d] & J(k) \ar[d] \ar[dr]^{r}\\
            C(\A_k) \ar[r]^-{\iota} & J(\A_k) \ar[r]^-{r} & \prod_{v \textup{ good}} J(\F_v) \ar[r]^-{\pr_{d^\infty}} &  \prod_{v \textup{ good}} J(\F_v)[d^\infty]\,.
            }
        \end{equation}
        Assuming finiteness of $\Sha(k,J)$, Scharaschkin showed \cite[Theorem 3.1.1]{Scharaschkin} that $\iota$ induces a bijection between $C(\A_k)^{\Br}$ and the intersection of $\iota(C(\A_k))$ with the topological closure of $J(k)$ in $J(\A_k)$. He also considered the intersection after applying the map $r$ as featured in the following definition. We extend this to consider the $d$-primary and prime-to-$d$ parts.
        
        \begin{Definition}
            The \defi{Poonen-Scharaschkin set} $C(\A_k)^{\PS}$ is {the subset of adelic points whose image under $r\circ\iota$ lands in }the topological closure (with respect to the {product of the} discrete topologies) of $r(J(k))$ inside $\prod_{v \textup{ good}} J(\F_v)$. Similarly, the \defi{$d$-primary Poonen-Scharaschkin set} $C(\A_k)^{\PS[d^\infty]}$ is {the subset of adelic points whose image under $\pr_{d^\infty} \circ r \circ \iota$ lands in } the topological closure of $\pr_{d^\infty}(r(J(k)))$ inside $\prod_vJ(\F_v)[d^\infty]$.
        \end{Definition}
        
        \begin{Proposition}\label{prop:PS} Suppose $\iota\colon C \to J$ is as above and $d$ is an integer.
            \begin{enumerate}
                \item If $\Sha(k,J)[d^\infty]$ is finite, then $C(\A_k)^{\Br C[d^\infty]} \subset C(\A_k)^{\PS[d^\infty]}$.
                \item If $\Sha(k,J)[d^\perp]$ is finite, then $C(\A_k)^{\Br C[d^\perp]} \subset C(\A_k)^{\PS[d^\perp]}$.
            \end{enumerate}
        \end{Proposition}
        
        \begin{Remark}\label{rem:Scharaschkin}
            This proposition is a slightly refined version of Scharaschkin's result \cite{Scharaschkin} that if $\Sha(k,J)$ is finite and $C(\A_k)^{\PS} = \emptyset$, then $C(\A_k)^{\Br C} = \emptyset$.
        \end{Remark}

        \begin{proof}
            It suffices to prove the first statement for $d$ a supernatural integer. Let $(P_v) \in C(\A_k)^{\Br C [d^\infty]}$. We will show that, for any finite set of places $S$ of good reduction for $J$, there is some $Q \in J(k)$ such that $(P_v)$ and $Q$ have the same image in $\prod_{v\in S}J(\F_v)[d^\infty]$. Let $e$ be the exponent of the finite group $\prod_{v \in S}J(\F_v)[d^\infty]$. By assumption $(P_v) \in C(\A_k)^{\Br C [e^\infty]}$. Since $\Sha(k,J)[e^\infty] \subset \Sha(k,J)[d^\infty]$ is assumed to be finite, Corollary~\ref{cor:d-div} yields the existence of some $Q \in J(k)$ such that $(\iota(P_v) - Q) \in eJ(\A_k)$. Then $(P_v)$ and $Q$ have the same image in $\prod_{v\in S}J(\F_v)[d^\infty]$, since $e$ is the exponent of this group.
        \end{proof}

        In the case of constant curves over global function fields, we have the following partial converse to Proposition~\ref{prop:PS}.

        \begin{Proposition}\label{prop:PSconverse}
            Suppose $C$ is a constant curve over a global function field $k$ and $d \ge 1$ is an integer. 
            \begin{enumerate}
                \item If $C(\A_k)^{\Br C} = \emptyset$, then $C(\A_k)^{\PS}= \emptyset$. 
                \item If $C(\A_k)^{\Br C[d^\infty]} = \emptyset$, then $C(\A_k)^{\PS[d^\infty]} = \emptyset$.
            \end{enumerate}
        \end{Proposition}

        The proof will require the following lemma.
        \begin{Lemma}\label{lem:iterateFrob}
            Suppose $X$ is a nice constant variety over a global function field $k$ with constant field $\F_q$. Let $F\colon X \to X$ denote the $\F_q$-Frobenius. The sequence of functions $F^{n!} \colon X(\A_k) \to X(\A_k)$ converges uniformly to the ``reduction map'' $X(\A_k) \to X(\A_\F) := \prod X(\F_v) \subset X(\A_k)$.
        \end{Lemma}
        
        \begin{proof}
            It is clear that the pointwise limit of the sequence is the reduction map. The sequence is equicontinuous, since for any $x,y \in X(\A_k)$ we have $\rho(x,y) \ge \rho(F(x),F(y))$, where $\rho$ is any metric inducing the adelic topology. Uniform convergence then follows from the Arzel\'a-Ascoli theorem, which applies as $X(\A_k)$ is compact (e.g., \cite[Ch. 7, Lemma 39]{Royden}). 
        \end{proof}

        \begin{proof}[Proof of Proposition~\ref{prop:PSconverse}]
            Enumerate the places of $k$ and define  $S_m$ to be the set consisting of the first $m$ of them.  {Note that every place $v\in \Omega_k$ is a place of good reduction since $C$ is a constant curve. }Suppose $(P_v) \in C(\A_k)^{\PS}$. Then for each $m \ge 1$ there exist $Q_m \in J(k)$ such that $r_v(\iota(P_v)) = r_v(Q_m)$ in $J(\F_v)$, for all primes $v$ in $S_m$. By Lemma~\ref{lem:iterateFrob} the limit $\lim_{n}F^{n!}(Q_n)$ exists and is equal to $\lim_nF^{n!}(\iota(P_v)) = r(\iota(P_v)) \in J(\A_\F) \subset J(\A_k)$. Thus, $r(\iota(P_v)) \in \overline{J(k)} \subset J(\A_k)^{\Br}$.  {Since $r$ commutes with $\iota$, Corollary~\ref{cor:d-div} implies that $r(P_v)\in C(\A_k)^{\Br}$.}

            Now we prove the $d$-primary version. Assume $(P_v) \in C(\A_k)^{\PS[d^\infty]}$. Then there exist $Q_n \in J(k)$ such that $Q_n - \iota(P_v)$ reduces to $J(\F_v)[d^\perp]$ for all $v \in S_n$. By Lemma~\ref{lem:iterateFrob}, the limits $  r(\iota(P_v)) = \lim_nF^{n!}(\iota(P))$ and $Q_0 := \lim_n F^{n!}Q_n$ exist and are equal after projecting to the $d$-primary subgroup of $J(\A_\F) = \prod J(\F_v)$. So $r(\iota(P_v))-Q_0 \in J(\A_\F)[d^\perp] \subset J(\A_k)[d^\perp]$. In particular, $r(\iota(P_v))-Q_0 \in d^\infty J(\A_k)$. Also, $Q_0 \in \overline{J(k)} \subset J(\A_k)^{\Br}$. Since $\Sha(k,J)[d^\infty]$ is finite, Corollary~\ref{cor:d-div} gives, for every $n$, some $R_n \in J(k)$ such that $Q_0-R_n \in d^nJ(\A_k)$. Thus we see that $(r(\iota(P_v)) - R_n) = (r(\iota(P_v))-Q_0) - (Q_0 - R_n) \in d^nJ(\A_k)$. It now follows from Corollary~\ref{cor:d-div} and {the commutativity of $r$ and $\iota$ that $(r(P_v)) \in C(\A_k)^{\Br C[d^\infty]}$}.
        \end{proof}

    %%%%%%%%%%%%%%%%%%%%%%%%%%%%%%%%%%%%%%%%%%%%%%%%%%%%%%%%%%%%%%%%%%%%%%%%%%%%
    \subsection{Sufficiency of the Brauer-Manin obstruction on certain constant curves}
    %%%%%%%%%%%%%%%%%%%%%%%%%%%%%%%%%%%%%%%%%%%%%%%%%%%%%%%%%%%%%%%%%%%%%%%%%%%%

        \begin{Theorem}\label{thm:constant}
            Let  $C$ be a curve defined over a finite field $\F$ with $C(\F) = \emptyset$ and let $k/\F$ be a function field 
            such that  $J(k)$ is not Zariski dense in $J$, the Jacobian of $C$. Then $C(\A_k)^{\Br}= \emptyset$.
        \end{Theorem}
        
        \begin{Remark}
            Note that in \cite{PoonenVoloch} it is shown that the Brauer-Manin obstruction is 
            the only obstruction to weak approximation for curves over function fields satisfying additional,
            but quite general, hypotheses. One of these hypotheses, however, excludes the case of 
            constant curves. There is also an example in \cite{PoonenVoloch} showing that, on certain
            constant curves, there can be points in $C(\A_k) \cap {\overline {J(k)}}$ 
            which are not in $C(k)$ (though they are in ${\overline {C(k)}}$). Additionally, if $J(k)$ is finite, then the Brauer-Manin obstruction is the only obstruction to the Hasse principle, as shown by Scharaschkin \cite{Scharaschkin} for number fields, and
            the same argument applies to function fields. The above theorem, in the case of constant curves over function fields,
            extends these results of \cites{Scharaschkin,PoonenVoloch} {to cover some cases where the Jacobian is not simple.}
        \end{Remark}

        \begin{proof}
            By the Weil bounds, $C$ has a $0$-cycle of degree $1$. We may therefore view $C$ as a subvariety of $J$ as in Section~\ref{subsec:curves}. The Zariski closure $Z$ of $J(k)$ is a finite union of {torsion translates} of {(base changes of)} proper abelian subvarieties of $J$ defined over a finite field, so $Z \cap C$ is a $0$-dimensional scheme defined over a finite field. Therefore all points lying on this intersection are {points of $C$ that are} torsion in $J(\kbar)$ of order dividing some integer $n$.
            This, together with the assumption that $C(\F)=\emptyset$ already implies that $C(k) = \emptyset$.
        
            We now prove that the Brauer-Manin set is empty.  Recall that Tate and Milne have proved that the Tate-Shafarevich group of a constant abelian variety is finite.  (See Tate \cite{Tate-BSD} and Milne \cite[Theorem 8.1]{Milne1975}. The restriction that $p \ne 2$ in \cite{Milne1975} can be lifted thanks to \cite{Illusie1979}.)  Therefore, by Remark~\ref{rem:Akbar}, it suffices to show that $C(\A_k) \cap {\overline {J(k)}} = \emptyset$, where ${\overline {J(k)}}$ denotes the closure of $J(k)\subset J(\A_k)$ in the adelic topology. 
            % By Remark~\ref{rem:Akbar} (and the fact, proved by Tate and Milne, that the Tate-Shafarevich group of a constant abelian variety is finite. See Tate \cite{Tate-BSD} and Milne \cite[Theorem 8.1]{Milne1975}. The restriction that $p \ne 2$ in \cite{Milne1975} can lifted thanks to \cite{Illusie1979}.), it suffices to show that $C(\A_k) \cap {\overline {J(k)}} = \emptyset$, where ${\overline {J(k)}}$ denotes the closure of $J(k)\subset J(\A_k)$ in the adelic topology. 
            The adelic points on this intersection all have components in $Z \cap C$ so have order dividing $n$. So the points on this intersection belong to the torsion
            subgroup of $\overline{J(k)}$ but that is just the torsion on $J(k)$ itself by \cite[Lemma 3.7]{PoonenVoloch}. But, as we've seen above, $C(k) = \emptyset$ and the claim follows.
        \end{proof}

%%%%%%%%%%%%%%%%%%%%%%%%%%%%%%%%%%%%%%%%%%%%%%%%%%%%%%%%%%%%%%%%%%%%%%%%%%%%%%%%
\section{$d$-primary obstructions for infinitely many coprime $d$}
%%%%%%%%%%%%%%%%%%%%%%%%%%%%%%%%%%%%%%%%%%%%%%%%%%%%%%%%%%%%%%%%%%%%%%%%%%%%%%%%

In this section we show that it is possible for a curve of genus at least $2$ over a global field to have $d$-primary Brauer-Manin obstructions for infinitely many coprime $d$. This is very different than the case of genus one curves where the only obstructions come from the $I$-primary subgroup, where $I$ is the index of the curve (see ~\cite[Theorem 1.2]{CV17}).  Our proof also yields a construction of a hyperelliptic curve $C$ with an odd Brauer-Manin obstruction to the Hasse principle (see Section~\ref{sec:HyperellipticNoOddObstruction}).  Interestingly, the Albanese torsor $\Alb^1_X$ of a hyperelliptic curve $X$ can \emph{never} have an odd Brauer-Manin obstruction to the Hasse principle~\cite[Cor. 4.3]{CV17}.

 \begin{Theorem}\label{thm:GivenEInfinitelyManyd}
        There exists a curve $C/\Q$ such that $C(\A_\Q) \ne \emptyset$ and $C(\A_Q)^{\Br C[d^\infty]} = \emptyset$ for an infinite set of pairwise coprime integers $d$.
    \end{Theorem}

    \begin{Proposition}\label{prop:Rank0BMpSizeJac}
        Let $C$ be a curve of index $a1$ over a global field $k$ whose Jacobian $J$ has rank $0$.  Let $v$ be a finite place of good reduction for $C$, let $p = \Char \F_v$ and let $d := p\cdot\#J(\F_v)$. If {$\Sha(k, J)[d^{\infty}]$ is finite} and $C(\A_k)^{\Br} = \emptyset$, then $C(\A_k)^{\Br C[d^\infty]} = \emptyset$.
    \end{Proposition}

    \begin{proof}[Proof of Proposition~\ref{prop:Rank0BMpSizeJac}]
	Since the index of $C$ is $1$, we may fix an embedding $\iota\colon C\hookrightarrow J$.  We will assume that $C(\A_k)^{\Br C[d^{\infty}]} \neq \emptyset$ for otherwise the theorem trivially holds.  Let $(P_v)\in C(\A_k)^{\Br C[d^{\infty}]}$. Since $\Sha(k, J)[d^{\infty}]$ is finite,  Corollary~\ref{cor:d-div} implies that for each $n\geq 1$ there exists a $Q_n\in J(k)$ such that $(\iota(P_v) - Q_n) \in d^nJ(\A_k)$.  Since $J(k)$ is finite, there is some $Q\in J(k)$ such that $Q = Q_n$ for infinitely many $n$.  As $n$ goes to $\infty$, $d^nJ(k_v)$ converges to $0$ (by definition of $d$), so $\iota(P_v)$ converges to $Q\in J(k)$.  Thus $Q$ is in the image of $\iota$, so $C(k)$ and $C(\A_k)^{\Br}$ are nonempty.
    \end{proof}

\begin{Remark}
	It would be very interesting to determine if the hypothesis in Proposition~\ref{prop:Rank0BMpSizeJac} that $J(k)$ is finite is necessary. 
\end{Remark}

\begin{Remark}
	If the index $I$ of $C$ is greater than $1$ and $\Sha(k,J)$ is finite, then $C$ has an $I$-primary Brauer-Manin obstruction to the existence of rational points. Indeed, if $C$ is not locally soluble, then this holds trivially. If $C$ is locally soluble, then for each $i$, $\Pic^i(C) = \Pic^i(\Cbar)^{G_k} = \Alb^i_C(k)$. By~\cite[Cor. 4.3]{CV17} we have that $\Alb^1_C$ has an $I$-primary obstruction. Hence the same is true for $C$ by functoriality of the Brauer-Manin pairing. Example~\ref{ex:index2} below shows that in this situation it is still possible that $C$ has $d$-primary Brauer-Manin obstructions for integers relatively prime to $I$.
\end{Remark}

Theorem~\ref{thm:GivenEInfinitelyManyd} would follow from Proposition~\ref{prop:Rank0BMpSizeJac} once one knows the existence of infinitely many places of good reduction such that $(\Char\F_v)\cdot(\#J(\F_v))$ are pairwise coprime.  We will instead use the following variant of Proposition~\ref{prop:Rank0BMpSizeJac}, which allows us to rely on the weaker statement that there exists an elliptic curve $E$ with infinitely many places of good reduction such that $\#E(\F_v)$ are pairwise coprime.

    \begin{Proposition}\label{prop:CoverOfA_BMSizeOfJac}
        Suppose $\pi\colon C \to A$ is a finite morphism from the curve $C$ to an abelian variety $A$ such that $\Sha(k,A)$ and $A(k)$ are both finite. {If there exists a prime $v_0$ of good reduction for $C$ and $A$ such that} 
        {$\pi(C(\F_{v_0}))\cap (A(k) \bmod v_0) = \emptyset$,}
        then $C(\A_k)^{\Br C[d^\infty]} = \emptyset$ where $d = \#A(\F_{v_0})$.
    \end{Proposition}
    \begin{proof}
        Suppose $(P_v) \in C(\A_k)^{\Br C[d^\infty]}$. Then $\pi(P_v) \in A(\A_k)^{\Br A[d^\infty]}$ and so, by Corollary~\ref{cor:AVdiv}, we see that for any positive integer $n$ there exists $Q_n \in A(k)$ and $(R_v^{(n)}) \in A(\A_k)$ such that $\pi(P_v) - Q_n = d^n(R_v^{(n)})$. Since $A(k)$ is finite, we may assume $Q_n = Q \in A(k)$ is constant. {Since $d = \#A(\F_{v_0})$, we see that the reduction of $P_{v_0}$ is {an $\F_{v_0}$-point on $C$ that maps to $Q \bmod v_0$, resulting in a contradiction.}}
    \end{proof}	
   
	\begin{Lemma}\label{lem:LocallySolvableCover}
        	Let $E$ be an elliptic curve over a global field $k$ of characteristic not $2$ with $E(k) = \{O\}$, let $\phi\colon E \to \PP^1 = \PP^1_{(x:z)}$ be a map given by $|2O|$ that sends $O$ to $\infty$, and let $S$ be a finite set of places.  If $q\in k$ is a sufficiently large, totally positive, principal prime that is congruent to $1$ modulo $8\OO_k$, then there exists a nice curve $C/k$ admitting a double cover $\pi\colon C \to E$ such that 
	        \begin{enumerate}
	            \item $\pi$ is branched over the points $\phi^*(V(x^2 - qz^2))$ and unramified at all other points, \label{cond:BranchLocus}
	            \item  $C(\A_k) \neq \emptyset$, and
	            \item $\kk(\pi^{-1}(O))$ is {equal to $k(\sqrt{p})$ for some totally positive principal prime $p$ congruent to $1$ modulo $8\OO_k$ with $p \notin S$}.\label{cond:NontrivialExtension}
	        \end{enumerate}
	        In particular, the set of places $v_0\in \Omega_k$ such that the hypotheses of Proposition~\ref{prop:CoverOfA_BMSizeOfJac} hold has Dirichlet density $1/2$.
    \end{Lemma}

	{\begin{Remark}
		In the case that $k$ is a global function field, the conditions that $q$ is totally positive and $q \equiv 1 \bmod 8\calO_k$ hold trivially. Indeed, {the positivity condition vacuously holds since there are no archimedean places} and $8$ is either a unit or $0$.
	\end{Remark}}

    \begin{proof}
        A double cover $\pi\colon C\to E$ is determined, up to quadratic twist, by its branch locus. Let $\pi\colon C \to E$ be a double cover of $E$ branched over $\phi^*(V(x^2 - qz^2))$ with the fiber above $O$ split. For any $a\in k^{\times}/k^{\times2}$, let $\pi^a\colon C^a\to E$ be the associated twist.
        
        By Riemann-Hurwitz, $C^a$ has genus $3$.  Thus, by Hensel's Lemma and the Hasse-Weil bounds, $C^a(k_v)\neq \emptyset$ for all finite $v$ with $\#\F_v \geq 37$ and such that $(C^a)_{\F_v}$ is geometrically irreducible.  Recall that a double cover is geometrically reducible if and only if the base is geometrically reducible or the divisor defining the double cover (in the case of $\pi$, this is $\phi^*(V(x^2 - qz^2)) - 4O$) is twice a principal divisor. Therefore, the places $v$ where $(C^a)_{\F_v}$ is geometrically reducible are the places $v$ such that $E_{\F_v}$ is geometrically reducible, the places that ramify in $k(\sqrt{q})$, and the places that ramify in $k(\sqrt{a})$.  
        
        Let $p$ be a principal prime that is totally positive and congruent to $1$ modulo $8\calO_{k}$, that is a square modulo all places $v$ with $\#\F_v\leq 37$, that is a square modulo all places $v$ where $E \bmod v$ is geometrically reducible, that splits completely in $\kk(\phi^*(V(x^2 - qz^2)))$, and that is not in $S$.  Such a $p$ exists, because all but the last condition can be imposed by asking that $p$ split in a particular field extension, so Chebotarev's density theorem says that there are in fact infinitely many such $p$.  
        
        It is clear from the construction that $C^p$ satisfies conditions~\eqref{cond:BranchLocus} and~\eqref{cond:NontrivialExtension}.  We claim that for such $p$, we also have $C^p(\A_k) \neq \emptyset$.  As explained above, it suffices to check $C^p(k_v) \neq \emptyset$ for (i) $v\mid\infty$, (ii) $v$ with $\#\F_v\leq 37$, (iii) $v$ where $E_{\F_v}$ is geometrically reducible, (iv) places $v$ that ramify in $k(\sqrt{q})$ and (v) places that ramify in $k(\sqrt{p})$.  Since $p$ is a square for all places of type (i), (ii) or (iii) and $C(\A_k) \neq\emptyset$ (in fact, $C(k) \neq \emptyset$), $C^p(k_v)\neq\emptyset$ for places of this type.  Since $q$ and $p$ are totally positive principal primes that are congruent to $1 \bmod 8\calO_k$, the only places that ramify in $k(\sqrt{q})$ and $k(\sqrt{p})$ are $v_q$ and $v_p$.  Since $v_p$ splits completely in $\kk(\phi^*(V(x^2 - qz^2)))$, there is a $k_{v_p}$-point in the branch locus of $\pi^a$ for any twist $a\in k^{\times}/k^{\times2}$ and hence a $k_{v_p}$-point on $C^p$.  Furthermore, since $q$ is a square modulo $p$ and $q$ and $p$ are both totally positive primes and $1$ modulo $8\calO_k$, the product formula applied to the Hilbert symbol $(q,p)$ implies that $p$ is a square modulo $q$, so we also have that $C^p(k_{v_q}) \neq \emptyset$.  Hence, $C^p(\A_k) \neq \emptyset$ as desired.
    \end{proof}

    \begin{proof}[Proof of Theorem~\ref{thm:GivenEInfinitelyManyd}]
        Let $E/\Q$ be an elliptic curve with $E(\Q)$ trivial, $\Sha(\Q, E)$ finite and $\Gal(\Q(E[q])/\Q) \simeq \GL_2(\F_q)$ for every prime $q$. For example, one may take the curve with minimal Weierstrass equation $y^2 + y = x^3 + x^2 - 12x - 21$, \cite[\href{http://www.lmfdb.org/EllipticCurve/Q/67/a/1}{Elliptic Curve 67.a1}]{LMFDB}.

        By Serre's open image theorem \cite{Serre72}, there is a finite set of primes $S$ such that for all $q\notin S$, $\Q(E[q])$ is linearly disjoint from $\Q(E[n])$ for any integer $n$ with $q\nmid n$. Let $\pi \colon C \to E$ be a locally soluble double cover given by Lemma~\ref{lem:LocallySolvableCover}. Let $L = \Q(\sqrt{p})$ be the quadratic extension of $\Q$ given by the residue field of $\pi^{-1}(O)$. By Lemma~\ref{lem:LocallySolvableCover}, $p$ is a positive prime outside $S$ that is congruent to $1$ modulo $8$. By Proposition~\ref{prop:CoverOfA_BMSizeOfJac} $C(\A_{\Q})^{\Br C[d^\infty]} = \emptyset$ for $d = \#E(\F_v)$ where $v$ is any finite place in the set
        \[
           T := \{v\in \Omega_{\Q} \;:\; E_{\F_v} \textup{ smooth, and }O\notin \pi(C(\F_{v}))\}.
        \]
	
	    We will inductively construct a sequence of primes in $T$ such that the corresponding integers $\#E(\F_v)$ are pairwise coprime.  Since $p \equiv 1 \bmod 8$, $L$ is contained in $\Q(\zeta_p)$ and hence $\Q(E[p])$. Assume that there exist $v_1, \dots, v_i\in T$ such that the integers $\#E(\F_{v_j})$ are pairwise coprime and such that $p\nmid \#E(\F_{v_j})$ for all $j$. Note that this assumption is satisfied in the base case $i = 1$ by the Chebotarev density theorem, since $\Gal(\Q(E[p])/\Q) \simeq \GL_2(\F_p)$. Let $n$ be the radical of the product $\prod_{j=1}^i \#E(\F_{v_i})$.  By the definition of $S$ and the assumption that $p\nmid n$, the pair of fields $\Q(E[n])$ and $\Q(E[p])$ are linearly disjoint.  This also implies that $\Q(E[n])$ and $L$ are linearly disjoint.

	    Since the Galois representations on $E[q]$ are surjective for all primes $q \mid n$, there exists some element $\sigma\in\Gal(E[n]/E)$ that fixes no nontrivial point of $E[n]$. Furthermore, since $\Gal(\Q(E[p])/\Q) \simeq \GL_2(\F_p)$, there exists an element $\tau\in \Gal(\Q(E[p])/\Q)$ such that $\tau|_L$ is nontrivial and $\tau$ fixes no nontrivial point of $E[p]$. (It suffices that $1$ is not an eigenvalue {and} that the determinant is not a square.)  Thus, the Chebotarev density theorem applied to the compositum of $\Q(E[p])$ and $\Q(E[n])$ yields a finite place $v_{i+1}\in T$ (since $v_{i+1}$ is inert in $L$) such that $\#E(\F_{v_{i+1}})$ is relatively prime to $np$.
    \end{proof}

    \begin{Remarks}\hfill
        \begin{enumerate}

            \item We expect that most curves over arbitrary global fields satisfy the conclusions of Theorem \ref{thm:GivenEInfinitelyManyd}. On the other hand, in the next section, we will show that this is not the case for \emph{every} curve.

            \item Over $\Q$ one might even expect that we can take all $d$ to be prime in Theorem~\ref{thm:GivenEInfinitelyManyd} for suitable curves, which would follow, e.g., from a conjecture of Koblitz \cite{Koblitz} that there are many elliptic curves for which 
 $\#E(\F_v)$ is prime infinitely often.

            \item We can also obtain a very different example satisfying the conclusion of  Theorem~\ref{thm:GivenEInfinitelyManyd} by applying Lemma~\ref{lem:LocallySolvableCover} to $E: y^2=x^3-x+2$ over $k=\F_3(t)$. Then $E$ is defined over $\F_3$ and $\# E(\F_3) = 1$. Given any integer $n>1$, if $f = [\F_3(E[n]):\F_3]$ and $\ell$ is a prime number such that  $(\ell,f)=1$, then $(\# E(\F_{3^\ell}), n) =1$. So, by picking successively suitable $n,\ell,v$ we get a list of places inert in $L$ with $\# E(\F_v)$ pairwise coprime.
        \end{enumerate}
    \end{Remarks}

    %%%%%%%%%%%%%%%%%%%%%%%%%%%%%%%%%%%%%%%%%%%%%%%%%%%%%%%%%%%%
    \subsection{A hyperelliptic curve with an odd order Brauer-Manin obstruction}\label{sec:HyperellipticNoOddObstruction}
    %%%%%%%%%%%%%%%%%%%%%%%%%%%%%%%%%%%%%%%%%%%%%%%%%%%%%%%%%%%%
    We close this section with the following application of Proposition~\ref{prop:CoverOfA_BMSizeOfJac}, which gives a curve $C$ of index $2$ with odd order Brauer-Manin obstructions. This is particularly interesting because $C \subset \Alb^1_C$ and  the Brauer-Manin obstruction on $\Alb_C^1$ is completely captured by the $2$-primary subgroup (cf. \cite[Thm 1.2 and Cor. 4.5]{CV17}).

    \begin{Lemma}\label{lem:HypEx}
        Let $C/k$ be a hyperelliptic curve over a global field $k$ of characteristic not equal to $2$ with affine model $y^2 = f(x)$. Suppose $\Sha(k,J)$ is finite and $J(k) = \{O\}$ is trivial. For any prime $v$ of good reduction for $J$ such that $f(x)$ has no linear factor mod $v$ we have $C(\A_k)^{\Br C[d^\infty]} = \emptyset$ where $d = \#J(\F_v)$.
    \end{Lemma}

    \begin{proof}
        There is a map $\phi\colon \Alb^1_C \to \Alb^0_C = J$ defined by sending the class of a degree $1$ zero-cycle $z$ to the class of $2z - (\infty^++\infty^-)$, where $\infty^{\pm}$ are the points at infinity on the affine model defining $C$. Let $\pi \colon C \to J$ be the composition of $\phi$ with the canonical map $C \to \Alb^1_C$. The result follows from Proposition~\ref{prop:CoverOfA_BMSizeOfJac} applied to the map $\pi$, since the fiber above $O$ is defined by $f(x) = 0$.
    \end{proof}

    \begin{Example}\label{ex:index2}
        The hyperelliptic curve $C/\Q$ defined by $y^2 = 7(x^6 + 2x^4 + x^2 + 2x + 2)$ has index $2$, $C(\A_\Q) \ne \emptyset$, and $C(\A_\Q)^{\Br C[\ell^\infty]} = \emptyset$ for $\ell = 3,5,7$ and possibly infinitely many other primes. 
    \end{Example}

    \begin{proof}
        To ease notation write $J^1 := \Alb^1_C$. Using Magma one can perform a $2$-descent on $J$ to determine that $\operatorname{Sel}^2(k,J) \simeq \Z/2\times \Z/2$. A further $2$-descent on $J^1$ as described in \cite{CreutzANTSX} gives that $J^1(\Q) = \emptyset$ and $\Sha(k,J)[2] \simeq \Z/2\times \Z/2$. Since $C$ is locally soluble, the fact that $J^1(\Q) = \emptyset$ proves that $C$ has no $0$-cycle of degree $1$, so $C$ has index $2$. $J$ has good reduction at $v = 3,5$ and $17$ where the orders of $J(\F_v)$ are $9$, $25$ and $343$, respectively. This shows that $J(\Q)$ is trivial. So Lemma~\ref{lem:HypEx} gives that there is a $d$-primary Brauer-Manin obstruction for every $d = \#J(\F_v)$ such that $f(x)$ has no linear factor modulo these primes. This applies to $v = 3,5,17$.
    \end{proof}

%%%%%%%%%%%%%%%%%%%%%%%%%%%%%%%%%%%%%%%%%%%%%%%%%%%%%%%%%%%%%%%%%%%%%%%%%%%%%%%%
\section{Degrees do not capture the Brauer-Manin obstruction on curves}
%%%%%%%%%%%%%%%%%%%%%%%%%%%%%%%%%%%%%%%%%%%%%%%%%%%%%%%%%%%%%%%%%%%%%%%%%%%%%%%%

    In this section we construct examples of pointless curves over global fields that have index $1$ and such that there are primes $\ell$ such that $\ell$-primary subgroup of the Brauer group does not obstruct the existence of rational points. For any $\ell$, a curve of index $1$ has projective embeddings of degree $\ell^m$ provided $m$ is sufficiently large. So these examples show, in particular, that degrees do not capture the Brauer-Manin obstruction to the existence of rational points on curves in the sense of \cite[Question 1.1]{CV17}.

    \begin{Theorem}\label{thm:Brodd}
        There exists a genus $3$ curve $C/\Q$ with index $1$ such that $C(\A_\Q)^{(\Br C)_{\textup{odd}}} \ne \emptyset$ and $C(\A_\Q)^{\Br C[2]} = \emptyset$.
    \end{Theorem}

    \begin{Theorem}\label{thm:Brell}
        For any prime number $\ell$, there exists a global field $k$ and a genus $2$ curve $C/k$ of index $1$ such that $C(\A_k)^{\Br C[\ell^\infty]} \ne \emptyset$ and $C(\A_k)^{\Br} = \emptyset$.
    \end{Theorem}

    \begin{Remark}
        In the case $\ell = 2$, this second example shows that even the canonical degree does not capture the Brauer-Manin obstruction to the existence of rational points on curves.
    \end{Remark}

	\begin{Remark}\label{rem:towers}
		If $C(\A_k)^{\Br C[\ell^\infty]} \ne \emptyset$, then by Corollary~\ref{cor:curved-ab} there is, for every $n \ge 1$, an adelic point of $C$ which lifts to an $\ell^n$-covering of $C$. If $C(k) = \emptyset$, these coverings are counterexamples to the Hasse principle. For the curve $C$ constructed in the proof of Theorem~\ref{thm:Brodd}, the pullback of multiplication by any odd integer $n$ on $J$ (with no twisting) gives a covering which is a counterexample to the Hasse principle. Thus one has an infinite tower of non-isomorphic abelian \'etale coverings of $C$ in which all covers are counterexamples to the Hasse principle.
	\end{Remark}

    The idea behind these constructions is to find curves embedded in their Jacobians so that they contain, everywhere locally, nontrivial torsion points of order prime to $\ell$. This gives adelic points that are $\ell^\infty$-divisible, so by the results in Section~\ref{subsec:curves} this gives points in $C(\A_k)^{\Br C[\ell^\infty]}$.  Section~\ref{subsec:AdelicTorsion} contains the core of this idea and the proof of Theorem~\ref{thm:Brodd} and Section~\ref{subsec:ConstantAndBrell} focuses on the case of constant curves, which is used to prove Theorem~\ref{thm:Brell} therein.

    %%%%%%%%%%%%%%%%%%%%%%%%%%%%%%%%%%%%%%%%%%%%%%%%%%%%%%%%%%%%%%%%%
    \subsection{Adelic torsion packets}\label{subsec:AdelicTorsion}
    %%%%%%%%%%%%%%%%%%%%%%%%%%%%%%%%%%%%%%%%%%%%%%%%%%%%%%%%%%%%%%%%%
    Let $C$ be a nice curve over a global field $k$. Define an equivalence relation on $C(\kbar)$ by declaring that $x \sim y$ if some multiple of the divisor $x-y$ is principal. The equivalence classes under this relation are called \defi{torsion packets}. Let us say that a subset $T$ of some torsion packet is an \defi{$n$-torsion packet} if $n(t-s)$ is principal for all $t,s \in T$. 

    \begin{Proposition}\label{lem:torsionpacket2}
        Suppose there is an $n$-torsion packet $T$ on $C$ and there is a zero-cycle of degree $1$ supported on $T$. Then $T(\A_k) \subset C(\A_k)^{\Br C[d^\infty]}$ for any $d$ relatively prime to $n$.
    \end{Proposition}

    \begin{proof}
    	{Let $z$ be a $0$-cycle of degree $1$ supported on $T$ and let $\iota:C \to J$ be the corresponding embedding given by $x \mapsto [x-z]$.} For $t \in T(\kbar)$, $\iota(t) = t - z$ is a sum of differences of pairs $t_i - t_j$ with $t_i,t_j \in T(\kbar)$. Each difference $t_i - t_j$ lies in $J[n]$ since $T$ is an $n$-torsion packet. So $\iota(t) \in J[n]$. Thus, if $(Q_v) \in T(\A_k)$, then $\iota(Q_v) \in J(\A_k)[n] \subset d^\infty J(\A_k)$. Corollary~\ref{cor:d-div} then shows that $(Q_v) \in C(\A_k)^{\Br C[d^\infty]}$.
    \end{proof}

    \begin{proof}[Proof of Theorem \ref{thm:Brodd}]
        {Let $T$ be the $0$-dimensional subscheme of $\mathbb{A}^1_{\Q}$ that is given by the vanishing of $f(x) = (x^2+3)(x^3-19)$; note that $T$ is a counterexample to the Hasse principle.}
    %Choose a reduced $0$-dimensional scheme $T$ that is a counterexample to the Hasse principle. Then $T$ may be embedded in $\mathbb{A}^1$. For example take $T : f(x) = 0$, where $f(x) = (x^2+3)(x^3-19) \in \Q[x]$. \bianca{edit this} 
    For any square free polynomial $g(x) \in \Q[x]$ relatively prime to $f(x)$, the set of Weierstrass points of the hyperelliptic curve $C : y^2 = f(x)g(x)$ is a $2$-torsion packet containing $T$ and the support of a zero-cycle of degree $1$. By Proposition~\ref{lem:torsionpacket2} we have that $T(\A_{\Q}) \subset C(\A_{\Q})^{(\Br C)_{\textup{odd}}}$ and so $C(\A_{\Q})^{(\Br C)_{\textup{odd}}} \ne \emptyset$. It remains to show that $g(x)$ can be chosen so that $C(\A_{\Q})^{\Br C[2]}= \emptyset$. For this one can take $g(x) = 2(x^3 + x + 1)$. A $2$-descent (implemented in Magma based on \cite{BruinStoll}) shows that $C(\Q) = \emptyset$. Moreover, by Corollary~\ref{cor:curved-ab}, we have that the set $C(\A_{\Q})^{2-\textup{ab}}$ of adelic points on $C$ that survive a $2$-descent is equal to set cut out by $\Br C[2]$. Hence $C(\A_{\Q})^{(\Br C)[2]}= \emptyset$ as desired.
    \end{proof}

 %%%%%%%%%%%%%%%%%%%%%%%%%%%%%%%%%%%%%%%%%%%%%%%%%%%%%%%%%%%%%%%%%%%%%%%%%%%%
	\subsection{Constant curves over function fields}\label{subsec:ConstantAndBrell}	%%%%%%%%%%%%%%%%%%%%%%%%%%%%%%%%%%%%%%%%%%%%%%%%%%%%%%%%%%%%%%%%%%%%%%%%%%%%

		For constant curves over global function fields the existence of local torsion points is guaranteed for all but finitely many primes. The following lemma ensures the existence of local torsion points of order prime to $\ell$, for any prime $\ell$. 

		\begin{Lemma}
		%\Brendan{do we need $\ell$ not divisible by characteristic for this?}
		\label{lem:finite-field}
		Let $C_0/\F_q$ be a curve of genus $g>1$ embedded in its Jacobian $J_0$. Let $\ell$ be a prime number and
		write $\# J_0(\F_{q^n}) = \ell^{r_n} m_n, \ell \nmid m_n$.  If $q^n > (2g\ell^{r_n})^2$, then there exists a point 
		in $C_0(\F_{q^n})$ of order prime to $\ell$ in $J_0$. Moreover, the inequality $q^n > (2g\ell^{r_n})^2$
		will hold for all $n$ sufficiently large.
		\end{Lemma}

		\begin{proof}
		For the first part of the theorem, it is enough to consider $n=1,r = r_1$ and assume that $q > (2g\ell^r)^2$.
		There exists an abelian variety $A/\F_q$ and a separable isogeny $\phi\colon A \to J_0$
		of degree $\ell^r$ for which $\phi(A(\F_q))$ is the subgroup of $J_0(\F_q)$
		of points of order prime to $\ell$. To prove the lemma it is enough
		to show that the curve $\phi^*(C_0)$ has an $\F_q$-rational point.
		The genus $g'$ of $\phi^*(C_0)$ satisfies $2g'-2 = \ell^r(2g-2)$ and the
		statement of the lemma now follows from the Weil bound applied to $\phi^*(C_0)$.

		For the second part, recall that there exist algebraic integers 
		$\al_1,\ldots,\al_{2g}$ with $|\al_i| = q^{1/2}$,
		the eigenvalues of Frobenius, such that $\# J_0(\F_{q^n}) = \prod (\al_i^n -1)$.
		As the $\al_i$ are not roots of unity, whenever they are $\ell$-adic units,
		the multiplicative order of $\al_i$ modulo
		$\ell^r$ grows exponentially with $r$.  Therefore the smallest $n$ such that
		$\# J_0(\F_{q^n})$ is divisible by $\ell^r$ grows exponentially with $r$
		and so satisfies $q^n > (2g\ell^r)^2$ with finitely many exceptions.
		\end{proof}

		\begin{Proposition}\label{prop:Brell} 
		Let $C_0/\F_q$ and $\ell$ be as as in Lemma~\ref{lem:finite-field},
		and $k/\F_q$ a function field in one variable such
		that for all  $n \leq 2\log_q(2g\ell^{r_n})$ (using the notation from Lemma~\ref{lem:finite-field}), $k$ has no places of degree $n$. Then the curve $C := C_0 \times_{\F_q} k$ satisfies $C(\A_k)^{\Br C[\ell^\infty]} \ne \emptyset$.
		\end{Proposition}
		\begin{Remark}
			{By Lemma~\ref{lem:finite-field}, there are only finitely many $n$ such that $n \leq 2\log_q(2g\ell^{r_n})$.}
		\end{Remark}
		\begin{proof}
		Lemma~\ref{lem:finite-field} provides, for every place $v$ of $k$, an infinitely $\ell$-divisible point in $C$ with 
		coordinates in the residue field of $v$ so, a fortiori, in the completion $k_v$. The adelic point 
		with these points as coordinates belong to $C(\A_k)^{\Br C[\ell^\infty]}$ by Corollary~\ref{cor:d-div}.
		\end{proof}

		\begin{proof}[Proof of Theorem~\ref{thm:Brell}]
			Consider the curve $C_0: y^2 = -x^6+x^2-1$ over $\F_3$ and let $k_m$ be the function field of the curve $D_m: (x^{3^m}-x+1)(y^{3^m}-y-1)=1$ over $\F_3$. It is easy to check that $D_m(\F_{3^m})$ is empty. So for $m$ sufficiently large, Proposition~\ref{prop:Brell} applies to give that the curve $C := C_0 \times_{\F_3} k_m$ over $k_m$ has $C(\A_{k_m})^{\Br C[\ell^\infty]} \ne \emptyset$. The Jacobian $J_0$ of $C_0$ is isogenous to the product of two elliptic curves, one ordinary and other supersingular. Indeed, there is an obvious morphism to the supersingular elliptic curve given by $y^2 = -x^3 + x - 1$. {The other factor is the curve $y^2 = -x^3 + x^2 - 1$, which is ordinary.} %One can check that the Hasse-Witt matrix of $C_0$ is not nilpotent, which implies the other factor of the Jacobian is ordinary. 
			 By results of \cite{Subrao} the Jacobian of $D_m$ is ordinary. It follows that there can be no surjective map $D_m \to C_0$. But there is no constant map either because $C_0(\F_3)= \emptyset$. Hence $C(k_m)$ is empty. Moreover, any map $D_m \to J_0$ lands in a coset of the ordinary factor of $J_0$. So, the $k_m$-rational points of $J = J_0 \times k$ are not Zariski dense and Theorem \ref{thm:constant} gives that $C(\A_k)^{\Br} = \emptyset$.			
		\end{proof}

%%%%%%%%%%%%%%%%%%%%%%%%%%%%%%%%%%%%%%%%%%%%%%%%%%%%%%%%%%%%%%%%%%%%%%%%%%%%%%%%
\section{Poonen's Heuristic}
%%%%%%%%%%%%%%%%%%%%%%%%%%%%%%%%%%%%%%%%%%%%%%%%%%%%%%%%%%%%%%%%%%%%%%%%%%%%%%%%
    
	We remind the reader of the Poonen-Scharaschkin set $C(\A_k)^{\PS}$ defined in Section~\ref{subsec:PS}.
    Poonen conjectured \cite[Conjecture 5.1]{Poonen} that a pointless curve $C$ of genus at least $2$ over a number field $k$ has $C(\A_k)^{\PS} = \emptyset$, or, equivalently, that there exist a finite set $S$ of finite places of good reduction such that the images of $J(k)$ and $\prod_{v\in S} C(\F_v)$ in $\prod_{v\in S} J(\F_v)$ do not intersect. By Scharaschkin's result (cf. Remark~\ref{rem:Scharaschkin}), Poonen's conjecture implies that the Brauer-Manin obstruction is the only obstruction to the existence of rational points on curves over number fields. As evidence for the conjecture Poonen gives an analysis of the probability that the intersection is nonempty assuming:
    \begin{enumerate}
        \item the numbers $\#J(\F_v)$ are at least as smooth as typical integers of the same size, and
        \item the subsets $\iota(C(\F_v)) \subset J(\F_v)$ are modeled on random subsets of the same size, independent of $v$.
    \end{enumerate}
	Specifically, Poonen makes the following conjecture and derives the following consequence.
	\begin{Conjecture}[{\cite[Conj. 7.1]{Poonen}}]\label{conj:smooth}
		Let $A$ be an abelian variety over a number field $k$ with $S_A$ the set of primes of good reduction for $A$, and let $u \in (0,1)$. Then
			\[
				\limsup_{B\to \infty}\frac{\# \{v\in S_A : \#\F_v \leq B \textup{ and }\#A(\F_v) \textup{ is }B^u\textup{ smooth}\}}{\#\{v\in S_A: \#\F_v \leq B\}} > 0.
			\]
	\end{Conjecture}

	\begin{Theorem}[{\cite[Thm. 7.2]{Poonen}}]\label{thm:Poonen}
		Let $C$ be a curve of genus at least $2$ over a number field $k$ with Jacobian $J$ and let $S_C$ be the set of primes of good reductions for $C$. Assume Conjecture~\ref{conj:smooth} holds for $J$. 	For each prime $v\in S_C$, let $\calC_v$ be a random subset of $J(\F_v)$ of size $\#C(\F_v)$; assume that the choices for different $v$ are independent.  Then with probability $1$, there exists a finite subset $S\subset S_C$ such that the images of $J(k)$ and $\prod_{v\in S} \calC_v$ in $\prod_{v\in S} J(\F_v)$ do not intersect.
	\end{Theorem}

	In this section we show that the same assumptions and analysis lead to the expectation that a pointless curve $C$ over a number field of genus at least $2$ has $C(\A_k)^{\PS[d^\perp]} = \emptyset$, for any integer $d$ (See Theorem~\ref{thm:primetodHeuristic} below).  However, the curve $C/\Q$ produced by Theorem~\ref{thm:Brodd} provides a counterexample to this statement. Indeed, it has $C(\Q) = C(\A_\Q)^{\Br}= \emptyset$, but  $C(\A_\Q)^{\Br C[2^\perp]}\ne \emptyset$. Proposition~\ref{prop:PS} then implies that $C(\A_\Q)^{\PS[2^\perp]}\ne \emptyset$.
	
	It is fairly easy to see where the heuristic goes astray in this example. The curve has the property that $\iota(C(\A_\Q))$ contains elements of $J(\A_\Q)[2]$, implying that the subsets $\iota(C(\F_v)) \subset J(\F_v)$ all contain $2$-torsion points of $J(\F_v)$. This illustrates an important point that these subsets cannot be modeled by random subsets of the appropriate size even when $C$ has no rational points.

	\begin{Theorem}\label{thm:primetodHeuristic}
		Let $C$ be a curve of genus at least $2$ over a number field $k$ with Jacobian $J$, let $S_C$ be the set of primes of good reductions for $C$,and let $d \ge 1$ be an integer. Assume Conjecture~\ref{conj:smooth} holds for $J$. For each prime $v\in S_C$, let $\calC_v$ be a random subset of $J(\F_v)[d^\perp]$ of size $\#C(\F_v)$; assume that the choices for different $v$ are independent.  Then with probability $1$, there exists a finite subset $S\subset S_C$ such that the images of $J(k)$ and $\prod_{v\in S} \calC_v$ in $\prod_{v\in S} J(\F_v)[d^\perp]$ do not intersect.
	\end{Theorem}

	\begin{proof}
	Our proof of this theorem follows Poonen's argument proving Theorem~\ref{thm:Poonen}. Let $G$ be a finitely generated abelian group, $S$ a positive density set of finite places of $k$, $\{G_v\}_{v\in S}$ a collection of finite abelian groups with homomorphisms $G \to G_v$, and $\{C_v\}_{v\in S}$ a collection of random finite subsets where the choices for different $v$ are independent.  Poonen's argument shows that if $\#C_v = O(\#\F_v)$, $\#G_v = (\#\F_v)^{g+ o(1)}$ for some real number $g>1$ that is independent of $v$, and each $\#G_v$ is $b$-smooth where $b = \sqrt{\#\F_v}$, then with probability $1$ there exists a finite subset $S'\subset S$ such that the images of $G$ and $\prod_{v\in S'} C_v$ in $\prod_{v\in S'} G_v$ do not intersect. 

    We will apply this with $G = J(k)$ and $G_v = J(\F_v)[d^\perp]$. Conjecture~\ref{conj:smooth} implies that there is a positive density set $S_1$ of primes such that $J(\F_v)[d^\perp]$ is $b$-smooth for $b = \sqrt{\#\F_v}$. By the Weil conjectures $\#J(\F_v)[d^\perp] \cdot \#J(\F_v)[d^\infty]  = \#J(\F_v)= (\#\F_v)^{g+o(1)}$. Let $S_N$ denote the set of places such that $\#J(\F_v)[d^\infty] \ge N$. These sets have, by Chebotarev's density theorem, natural densities which we claim tend to $0$ as $N$ tends to $\infty$. It follows that there are sets of primes of density arbitrarily close to $1$ on which $J(\F_v)[d^\infty]$ is bounded. Hence there is a positive density set of places $S_2 \subset S_1$ such that $\#J(\F_v)[d^\perp] = (\#\F_v)^{g + o(1)}$ as required. Thus the proof of the claim follows from the following lemma, after observing that for each integer $d$ and each $v \nmid d$, the reduction map gives a surjective map $J(k_v)[d^\infty] \to J(\F_v)[d^\infty]$.
	\end{proof}

    \begin{Lemma}
        Let $A$ be an abelian variety over a number field $k$. The density of the set of places $v$ of $k$ such that $A(k_v)$ has a point of order $\ell^n$ tends to $0$ as $n$ tends to $\infty$.
    \end{Lemma}

    \begin{proof}
        Let $G_n \subset \operatorname{GSp}_{2\dim(A)}(\Z/\ell^n)$ denote the image of the representation of $\Gal(k)$ on $A[\ell^n]$ in some fixed basis, and let $G_{n,1}\subset G_n$ denote the subset of matrices with $1$ as an eigenvalue. The density of the set of places such that $A(k_v)$ has a point of order $\ell^n$ is equal to $\#G_{n,1}/\#G_n$ by Chebotarev's density theorem. By a result of Serre in \cite{Serre79} (see also \cite[Corollary 2.1.7]{McQuillan}), there is an integer $m$, depending on $A$ but independent of $n$, such that $G_n$ contains contains $(\Z/\ell^n)^{\times m}$ as a subgroup of scalar matrices. Therefore we have a map of sets $(\Z/\ell^n)^{\times m}\times G_{n,1} \to G_n$ sending $(\lambda,M)$ to $\lambda M$. If $N = \lambda M$ is in the image of this map, then $\lambda$ is an eigenvalue of $N$. Hence the preimage of $N$ has size at most $2\dim(A)$, which is independent of $n$. On the other hand, $\#(\Z/\ell^n)^{\times m}$ is unbounded as $n$ goes to infinity, so $\#G_{n,1}/\#G_n \to 0$ as $n \to \infty$.
    \end{proof}

%%%%%%%%%%%%%%%%%%%%%%%%%%%%%%%%%%%%%%%%%%%%%%%%%%%%%%%%%%%%%%%%%%%%%%%%%%%%%%%%
%%%%%%%%%%%%%%%%%%%%%%%%%%%%%%%%%%%%%%%%%%%%%%%%%%%%%%%%%%%%%%%%%%%%%%%%%%%%%%%%
%%%%%%%%%%%%%%%%%%%%%%%%%%%%%%%%%%%%%%%%%%%%%%%%%%%%%%%%%%%%%%%%%%%%%%%%%%%%%%%%
%%%%%%%%%%%%%%%%%%%%%%%%%%%%%%%%%%%%%%%%%%%%%%%%%%%%%%%%%%%%%%%%%%%%%%%%%%%%%%%%
%%%%%%%%%%%%%%%%%%%%%%%%%%%%%%%%%%%%%%%%%%%%%%%%%%%%%%%%%%%%%%%%%%%%%%%%%%%%%%%%

%----------------------------------------------------------------------
%\end{thebibliography}
%----------------------------------------------------------------------

\end{document}